\documentclass[12pt,a4paper]{amsart}

\newtheorem{theo+}              {Theorem}           [section]
\newtheorem{prop+}  [theo+]     {Proposition}
\newtheorem{coro+}  [theo+]     {Corollary}
\newtheorem{lemm+}  [theo+]     {Lemma}
\newtheorem{exam+}  [theo+]     {Example}
\newtheorem{rema+}  [theo+]     {Remark}
\newtheorem{defi+}  [theo+]     {Definition}
\def \r{\mbox{${\mathbb R}$}}

\newenvironment{theorem}{\begin{theo+}}{\end{theo+}}
\newenvironment{proposition}{\begin{prop+}}{\end{prop+}}
\newenvironment{corollary}{\begin{coro+}}{\end{coro+}}
\newenvironment{lemma}{\begin{lemm+}}{\end{lemm+}}

\usepackage{amsthm}
\theoremstyle{plain} \theoremstyle{remark}
\newtheorem{remark}{Remark}

\def\E{/\kern-1.0em \equiv }

\title{$f$-biharmonic hypersurfaces into a conformally flat space}
\author{Ze-Ping Wang, Li-Hua Qin and  Xue-Yi Chen}

\address{}
\thanks{ School of Mathematical Sciences, Guizhou
Normal University, Guiyang 550025,
People's Republic of China. E-mail:zepingwang@gznu.edu.cn (Z.-P Wang is the corresponding author),\;lihuaqin@gznu.edu.cn (Qin),\;xueyichen@gznu.edu.cn (Chen)}

\begin{document}

\title[$f$-biharmonic hypersurfaces ]{$f$-biharmonic hypersurfaces into a conformally flat space }

\date{10/28/2024}
\maketitle

\section*{Abstract}
\begin{quote}
{\footnotesize     We first study $f$-biharmonicity of totally umbilical  hypersurfaces in a generic Riemannian manifold and then prove that any totally umbilical proper $f$-biharmonic hypersurface in a nonpositively  curved manifold has to be noncompact.  We also explore $f$-biharmonicity of   totally umbilical hyperplanes in a conformally flat space. Secondly,  we construct  $f$-biharmonic surfaces  and  biharmonic conformal immersions of  the associated surfaces into a conformall flat 3-space and also give a complete classification of  $f$-biharmonic surfaces of nonzero constant mean curvature in 3-space forms. Finally,   we especially investigate $f$-biharmonicity of  hypersurfaces  into  a conformally flat space  of  negative sectional curvature.  We  show that any totally umbilical $f$-biharmonic surface of a 3-manifold with nonpositve sectional curvature is minimal whilst there are proper $f$-biharmonic $m$-dimensional submanifolds with $m\geq3$ and $m\neq4$ into  nonpositvely curved manifolds.}
\end{quote}

\section{Introduction and Preliminaries}
A {\bf biharmonic  map}  is a  map  $\varphi:(M, g)\to (N,
h)$ between Riemannian manifolds whose bitension field  vanishes identically ( see \cite{Ji}), i.e.,
\begin{equation}\label{BT1}
\tau_{2}(\varphi):={\rm
Trace}_{g}(\nabla^{\varphi}\nabla^{\varphi}-\nabla^{\varphi}_{\nabla^{M}})\tau(\varphi)
- {\rm Trace}_{g} R^{N}({\rm d}\varphi, \tau(\varphi)){\rm d}\varphi
\equiv0,
\end{equation}
where $\tau(\varphi):={\rm
Trace}_{g}\nabla {\rm d} \varphi=\sum\limits_{i=1}^m\nabla^{\varphi}_{e_i}{\rm d} \varphi(e_i)-{\rm d} \varphi(\nabla_{e_i}e_i)$ is the tension field  of  $\varphi$ and $R^{N}$ denotes the curvature operator of $(N, h)$ defined by
$$R^{N}(X,Y)Z=
[\nabla^{N}_{X},\nabla^{N}_{Y}]Z-\nabla^{N}_{[X,Y]}Z.$$

An {\bf $f$-biharmonic map} is one whose $f$-bitension field solves the PDEs ( see e.g., \cite{Lu,Ou3})
\begin{equation}\label{FBT1}
\tau_{2,f}(\varphi):=f\tau_{2}(\varphi)+(\Delta f)\tau(\varphi)+2\nabla^{\varphi}_{{\rm grad}f}\tau(\varphi)=0,
\end{equation}
where $\tau_2(\varphi)$ is the  bitension fields of $\varphi$ and  $f:M\to(0,+\infty)$.

 From the variational point of view, one finds that a biharmonic map is the critical point of  the bienergy functional $E_2(\varphi)= \frac{1}{2} {\int}_{M}
|\tau(\varphi)|^2\,{\rm d}v_g$ whilst an $f$-biharmonic map is the critical point of  the $f$-bienergy functional $E_{2,f}(\varphi)= \frac{1}{2} {\int}_{M}
f|\tau(\varphi)|^2\,{\rm d}v_g$.

It is well known that we call a map $\varphi$ between Riemannian manifolds  a harmonic map if and only if the tension field  of  $\varphi$ vanishes identically (see e.g., \cite{EL1}). Clearly,  any  harmonic map  is  a biharmonic map whilst a biharmonic map is an $f$-biharmonic map.
So, we see that $f$-biharmonic maps  generalize the concept of harmonic maps and biharmonic maps.
There are the following inclusion relations among them:

$\{Harmonic\; \;maps\} \;\subset\;\{Biharmonic\;\; maps\}$$\;\subset\;\{f$-$biharmonic\; \;maps\}.$\\

 We call  non-biharmonic $f$-biharmonic maps   {\bf proper $f$-biharmonic maps}.

Recall that a submanifold is called a biharmonic ( resp., an $f$-biharmonic) submanifold if the isometric immersion that defines the submanifold is a biharmonic map  ( resp., an $f$-biharmonic).
We know from \cite{Ou02} that the biharmonic equation for a hypersurface in a generic Riemannian manifold $(N^{m+1}, h)$ defined by an isometric immersion  $\varphi:(M^{m},g)\to (N^{m+1},h)$ may be stated as\\

{\bf Theorem I.} (\cite{Ou02})
Let $\phi:M^{m}\to N^{m+1}$ be a hypersurface with mean curvature vector field
$\eta=H\xi$. Then $\phi$ is biharmonic if and only if:
\begin{equation}\label{BHEq}\notag
\begin{cases}
\Delta H-H |A|^{2}+H{\rm
Ric}^N(\xi,\xi)=0,\\
 2A\,({\rm grad}\,H) +\frac{m}{2} {\rm grad}\, H^2
-2\, H \,({\rm Ric}^N\,(\xi))^{\top}=0.
\end{cases}
\end{equation}
Here, $A$ denotes the shape operator
of the hypersurface with respect to the unit normal vector field $\xi$ and ${\rm Ric}^N : T_qN\longrightarrow T_qN$ is the Ricci
operator of the ambient space defined by $\langle {\rm Ric}^N\,
(Z), W\rangle={\rm Ric}^N (Z, W)$.\\

  Adopting the same notations and sign convention as in Theorem I, the $f$-biharmonic equation for  hypersurfaces in
a generic Riemannian manifold $(N^{m+1},h)$ can be stated as
\begin{theorem}$($\cite{Ou5,Ou3}$)$\label{ff0}
 A hypersurface $\phi:(M^{m},g)\to (N^{m+1},h)$  with mean curvature vector field
$\eta=H\xi$  is $f$-biharmonic if and only if:
\begin{equation}\label{ff1}
\begin{cases}
\Delta(fH)-(fH)[|A|^{2}-{\rm
Ric}^N(\xi,\xi)]=0,\\
 A\,({\rm grad}\,(fH)) +(fH)[\frac{m}{2} {\rm grad}\, H
- \,({\rm Ric}^N\,(\xi))^{\top}]=0,
\end{cases}
\end{equation}
which is equivalent to  
\begin{equation}\label{f1}\notag
\begin{cases}
\Delta H-H |A|^{2}+H{\rm
Ric}^N(\xi,\xi)+H(\Delta f)/f+2({\rm grad}\ln f)H=0,\\
 2A\,({\rm grad}\,H) +\frac{m}{2} {\rm grad}\, H^2
-2\, H \,({\rm Ric}^N\,(\xi))^{\top}+2\, H \,A\,({\rm grad}\,\ln f)=0.
\end{cases}
\end{equation}
\end{theorem}
We earlier remarked that $f$-biharmonic maps are generalizations of biharmonic
maps. Naturally,  $f$-biharmonic submanifolds generalize the notion of biharmonic
submanifolds. 

The study of biharmonic
submanifolds focuses on a fundamental problem that is to
classify biharmonic submanifolds in certain model spaces, such as space forms, conformally flat spaces and homogeneous spaces. Though significant
progress was made,  the following  conjecture remain open in general cases.

{\bf Chen's Conjecture (\cite{CH}):} Any biharmonic submanifold of  Euclidean space $\r^n$ is harmonic (i.e., minimal).

Another interesting problem in the study of biharmonic submanifolds
is the following conjecture.
 
 {\bf The Generalized Chen's Conjecture (\cite{CMO}):} Any biharmonic submanifold in  $(N^n,h)$ with nonpositive sectional curvature  is minimal.

  For more examples  on biharmonic submanifolds, 
and  basic examples and properties of biharmonic maps, we refer the readers to \cite{BFO, BMO1, BMO2,LiO,Lu,LO,Oua, Ou1,Ou2,Ou3,Ou5,OC} and 
the  references therein.\\

We would like to point out that the generalized Chen’s conjecture
has been proved to be false  by constructing  proper
bihmarmonic hyperplanes in a five-dimensional conformally flat space of negative sectional curvature in \cite{Ou2} where 
the authors used this and the biharmonicity of the product maps to construct infinitely many examples of proper biharmonic submanifolds
into nonpositively curved Riemannian manifolds.  As a generalization  of biharmonic
submanifolds, we want to know if there exist $f$-biharmonic submanifolds in nonpositively curved Riemannian manifolds.
On the other hand,  it is a fact that there is an interesting relationship among $f$-biharmonicity,  biharmonicity and conformality from a 2-dimensional Riemannian manifold.  For instance, a surface (i.e., an isometric immersion) $ \phi : (M^2, g = \phi^{*}h)\to(N^n, h)$ is an $f$-biharmonic surface if and only if the conformal immersion $\phi : (M^2, \bar{g}=f ^{-1}g)\to (N^n, h)$ is a biharmonic map with  the conformal factor $\lambda=f^{\frac{1}{2}}$. Moreover, the induced metric of the surface $\phi^{*}h=g=\lambda^2\bar{g}=f\bar{g}$  (see \cite{Ou3}, Theorem 2.3 and Corollary 3.6, for details). Based on the above, it would be interesting to study  $f$-biharmonic submanifolds  in certain model spaces. For some  progress on $f$-biharmonic submanifolds,  and basic examples and properties of $f$-biharmonic maps, one refers the readers to \cite{Lu, Ou1, Ou3, Ou4, Ou5, WOY1,WOY2,WC,WQ} and the book \cite{OC}, and the references therein.

We call a Riemannian manifold $(N^n, h) $  a conformally flat space if for
any point of $N^n$ there is a neighborhood that is conformally diffeomorphic
to an open subset of the Euclidean space $\r^n$. Note that all 2-dimensional Riemannian manifolds, space forms, $S^n\times\r$ and $H^n\times\r$ are  conformally flat spaces.\\

In this paper, we study $f$-biharmonic hypersurfaces into a  conformally flat space.  Firstly,  we give a characterization of  totally umbilical $f$-biharmonic hypersurfaces into a generic Riemannian manifold ({\bf Theorem \ref{FS1}}) and then prove that any totally umbilical proper $f$-biharmonic hypersurface in a nonpositively  curved manifold has to be noncompact ({\bf Proposition \ref{FS3}}).  We also explore  $f$-biharmonicity of a family of totally umbilical hypersurfaces into  conformally flat spaces ({\bf Theorem \ref{PQ0} and Theorem \ref{PC1}}). 
Such hypersurfaces are constructed by starting with a totally geodesic hyperplane in a Euclidean space and then performing a suitable conformal change of the Euclidean metric into a conformally flat metric. Secondly, using the link among $f$-biharmonicity,  biharmonicity and conformality from a 2-dimensional Riemannian manifold, we study biharmonic conformal immersions and $f$-biharmonic surfaces into a conformally flat 3-space. We also give a complete classification of proper $f$-biharmonic surfaces of nonzero constant mean curvature into 3-space forms ({\bf Theorem \ref{cs}}). Finally,  we especially investigate $f$-biharmonicity of  hypersurfaces into  conformally flat $n$-spaces  with strictly negative sectional curvature. We  show that any totally umbilical $f$-biharmonic surface of a 3-manifold with nonpositve sectional curvature is minimal and there exist  proper $f$-biharmonic $m$-dimensional submanifolds with $m\geq3$ and $m\neq4$ in nonpositively curved manifolds ({\bf Corollary \ref{qlyx}}).

\section{ Totally umbilical $f$-biharmonic hypersurfaces}
In this section, we first study totally umbilical $f$-biharmonic  hypersurfaces  into a  Riemannian manifold. 
We then investigate  $f$-biharmonicity of a family of totally umbilical hypersurfaces into  conformally flat spaces.  Those hypersurfaces are constructed by starting with a totally geodesic hyperplane in a Euclidean space and then performing a suitable conformal change of the Euclidean metric into a conformally flat metric.
More specifically,  following \cite{Ou2,LiO}, we  study  $f$-biharmonicity of a totally umbilical hyperplane $\varphi : (\r^m,g=\varphi^*h) \to (\r^{m+1}, h )$ into a conformally flat space with $\varphi(x_1,\ldots,x_m) = (x_1,\ldots,x_m,\sum\limits_{i=1}^{m}a_ix_i+a_{m+1})$, where $a_i$ are constans.\\

We are ready to give   a  characterization of totally umbilical $f$-biharmonic  hypersurfaces  into a general Riemannian manifold.
\begin{theorem}\label{FS1}
 A  totally umbilical  hypersurface $\varphi:(M^{m},g)\to (N^{m+1},\bar{h})$  with mean curvature vector field
$\eta=H\xi$  is $f$-biharmonic if and only if one of the following  cases happens:\\
$(i)$ the hypersurface  is totally geodesic; \\
$(ii)$ $m=4$ or $H=constant\neq0$,  the hypersurface  is  biharmonic. In particular, if $H=constant\neq0$, then ${\rm Ric}^N(\xi,\xi)=|A|^2=mH^2>0$;\\
$(iii)$ $m\neq4$, $f=c|H|^{\frac{m-4}{2} }$ and $H$ is nonconstant satisfying the following $PDE$
\begin{equation}\label{fs1}
({\rm Ric}^N\,(\xi))^{\top}=(m-1){\rm grad}H,\;
{\rm Ric}^N(\xi,\xi)=mH^2-|H|^{\frac{2-m}{2}}\Delta(|H|^{\frac{m-2}{2}})
\end{equation}
on the hypersurface, where $c$ is a positive constant. Moreover, by totally umbilical property of the hypersurface, the first equation of (\ref{fs1})  holds naturally, and the hypersurface is proper $f$-biharmonic in this case. 
\end{theorem}

\begin{proof}
One takes an orthonormal frame $\{e_1,\ldots,e_m,\;\xi\}$ of a $(m+1)$-dimensional Riemannian
manifold $N^{m+1}$ adapted to the totally umbilical hypersurface $M^m$ with $\eta=H\xi$ being the mean curvature vector field.
Thus,  we conclude that $A(e_i) = He_i$, $i=1,2,\ldots m$,  $|A|^2 = mH^2$ and $
 A({\rm grad} H)=H{\rm
grad}\, H.$

By  a straightforward computation, we have the Ricci curvatures as follows
\begin{eqnarray}\label{Cod1}
\sum_{i=1}^{m}\langle R(e_i,e_j)e_i,\xi\rangle =\sum_{i=1}^{m}R(\xi,
e_i,e_i, e_j) =-{\rm Ric}(e_j,\xi),\;j=1,2,\ldots,m.
\end{eqnarray}
Since $\{e_i\}$ are principal
directions with principal curvature $H$,  one can check that
\begin{eqnarray}\label{Cod2}
(\nabla_{e_i}h)(e_j,e_i)&=&{e_i}(h(e_j,e_i))-h(\nabla_{e_i}e_j,e_i)-h(\nabla_{e_i}e_i,e_j)\\\notag
&=&-h(e_i,e_i)\langle\nabla_{e_i}e_j,e_i\rangle-h(e_j,e_j)\langle\nabla_{e_i}e_i,e_j\rangle\\\notag
&=&-H(\langle\nabla_{e_i}e_j,e_i\rangle+\langle\nabla_{e_i}e_i,e_j\rangle)=0,
\end{eqnarray}
and
\begin{equation}\label{Cod3}
\begin{array}{lll}
(\nabla_{e_j}h)(e_i,e_i)={e_j}(h(e_i,e_i))-h(\nabla_{e_j}e_i,e_i)-h(\nabla_{e_j}e_i,e_i)\\
=e_j(H)-2h(e_i,e_i)\langle\nabla_{e_j}e_i,e_i\rangle
=e_j(H),\;\;\;\;for\;j=1,2\ldots m,
\end{array}
\end{equation}
where the covariant derivative of the second fundamental form $h$ is defined by
$$(\nabla h)(X,Y,Z)=(\nabla_{X}h)(Y,Z)={X}(h(Y,Z))-h(\nabla_{X}Y,Z)-h(\nabla_{X}Z,Y).$$
We substract (\ref{Cod3}) from (\ref{Cod2}) and use (\ref{Cod1}) and the Codazzi equation
 for a hypersurface to have
\begin{eqnarray}\label{Cod4}
e_j(H)={\rm R}^N(\xi, e_i,e_j, e_i),\;j=1,2,\ldots,m.
\end{eqnarray}
It follows that $(m-1)e_j(H)={\rm Ric}(e_j,\xi),\; j=1, 2,\ldots m,$  and hence
\begin{equation}\label{Cod5}
\begin{array}{lll}
(m-1){\rm grad}H =({\rm Ric}^N\,(\xi))^{\top}.
\end{array}
\end{equation}
This implies that if  the hypersurface is  totally umbilical,  then the 1st equation  of (\ref{fs1})  holds naturally. 

Clearly, $H=0$ implies that  the hypersurface is totally geodesic. Hereafter, assume that $H\neq0$.
One substitutes  (\ref{Cod5}) into the 2nd equation of
(\ref{ff1}) and simplifies the resulting equation to obtain
\begin{equation}\label{Cod6}
{\rm grad}(\ln(f|H|^{\frac{4-m}{2}})) =0,\;i.e.,\;f|H|^{\frac{4-m}{2}}=c
\end{equation}
on the  hypersurface, where $c>0$ is a constant.  It follows that each of the cases $m=4$ or $H=constant\neq0$  implies  that $f$  has to be a constant. In this case, by Theorem \ref{ff0}, the hypersurface is $f$-biharmonic if and only if it is biharmonic. Moreover, in this case, if $H$ is a nonzero constant, then we use (\ref{ff1}) to get  ${\rm Ric}^N(\xi,\xi)=|A|^2=mH^2>0$.

Now, suppose that $m\neq4$ and $H$ is  nonconstant. Substituting (\ref{Cod6}) into the 1st equation of
(\ref{ff1}) and simplifying the result yields
\begin{equation}\label{Cod7}
{\rm Ric}^N(\xi,\xi)=mH^2-|H|^{\frac{2-m}{2}}\Delta(|H|^{\frac{m-2}{2}}).
\end{equation}
Summarizing the above results we obtain the theorem.
\end{proof}

As  immediate consequences of  Theorem \ref{FS1}, we  have 
\begin{corollary}\label{FS2}
 A totally umbilical  hypersurface $\varphi: M^{m}\to (N^{m+1},h)$  with mean curvature vector field
$\eta=H\xi$  is proper $f$-biharmonic if and only if
$m\neq4$, $f=c|H|^{\frac{m-4}{2} }$ and $H$ is nonconstant satisfying
\begin{equation}\label{fs2}
{\rm Ric}^N(\xi,\xi)=mH^2-|H|^{\frac{2-m}{2}}\Delta(|H|^{\frac{m-2}{2}}),\;({\rm Ric}^N\,(\xi))^{\top}=(m-1){\rm grad}H
\end{equation}
on the hypersurface, where $c>0$ is a constant.
\end{corollary}
\begin{corollary}(\cite{Ou5})\label{FS02}
Any totally umbilical $f$-biharmonic hypersurface in a hyperbolic space $H^{m+1}$ or a Euclidean space $\r^{m+1}$
 is  totally geodesic, a totally umbilical $f$-biharmonic hypersurface in a sphere $S^{m+1}$
is  a part of the small sphere $S^m(\frac{1}{\sqrt{2}})$ or a part of the great sphere $S^m \subset S^{m+1}$.
\end{corollary}
\begin{proof}
Let  $ M^{m}\to N^{m+1}(C)$  be a totally umbilical  $f$-biharmonic hypersurface of mean curvature vector field
$\eta=H\xi$ into space forms with constant sectional curvature $C=-1,0,1$. Hence, $({\rm Ric}^N\,(\xi))^{\top}=0$. It follows from Theorem \ref{FS1} that the  hypersurfaces are totally umbilical biharmonic hypersurfaces. It is well known that such biharmonic hypersurfaces in hyperbolic spaces and  Euclidean spaces  are totally geodesic (see e.g., \cite{ CH, Ji}), and  a totally umbilical biharmonic hypersurface in  $S^{m+1}$  is  a part of the small sphere $S^m(\frac{1}{\sqrt{2}})$ in $S^{m+1}$ or a part of the great sphere $S^m \subset S^{m+1}$ (\cite{CMO,BMO1}).
\end{proof}

\begin{remark}\label{1}
(i) Corollary \ref{FS02} was also proved in \cite{Ou5} with a different way.

(ii) By Theorem \ref{FS1} and Corollary \ref{FS2}, we find that any totally umbilical  4-dimensional hypersurface in $N^5$ can not be proper $f$-biharmonic,  and any totally umbilical   proper $f$-biharmonic  hypersurface has nonconstant mean curvature. Note that  a totally umbilical biharmonic  hypersurface $M^{m}$ in $N^{m+1}$ with $m\neq4$ must have constant mean curvature (\cite{OW,BMO1}).

It is interesting to compare  Lemma 3.1 in \cite{Ou2} and  Theorem 3.4 in \cite{LiO} which provided
many examples of  totally umbilical  proper biharmonic 4-dimensional hypersurfaces with nonconstant mean curvature into  conformally flat 5-spaces.
 
For $m=2$, Our Theorem \ref{FS1} recovers Theorem 2.1 in \cite{WQ}.
\end{remark}
Applying Corollary \ref{FS2}, we have
\begin{proposition}\label{FS3}
Any totally umbilical   proper $f$-biharmonic hypersurface into a nonpositively curved manifold  has to be a non-compact  $m$-diemensional manifold with $m\geq3$ and $m\neq4$.
\end{proposition}
\begin{proof}
 Let $\varphi : (M^m,g=\varphi^*h) \to (N^{m+1}, h)$ be a totally umbilical  hypersurface with mean curvature vector field
$\eta=H\xi$ into a nonpositively curved manifold. This implies that ${\rm Ric}^N(\xi,\xi)\leq0$.

 We show that if the hypersurface is proper $f$-biharmonic, then it has to be non-compact.  In fact, if the hypersurface is proper $f$-biharmonic, by  Corollary \ref{FS2}, then $|H|>0$ is nonconstant and $m\neq4$. 
 Combining these,  a straightforward computation using (\ref{fs2}) yields
 \begin{equation}\label{fs21}
\Delta(|H|^{\frac{m-2}{2}})=m|H|^{\frac{m+2}{2}}-|H|^{\frac{m-2}{2}}{\rm Ric}^N(\xi,\xi)\geq m|H|^{\frac{m+2}{2}}>0
\end{equation}
on the hypersurface. Clearly, for $m=2$,  using (\ref{fs21}) we have $0=\Delta(|H|^{\frac{m-2}{2}})\geq m|H|^{\frac{m+2}{2}}>0$ which is a contradiction. For $m\geq3$, we  conclude from (\ref{fs21}) that $|H|^{\frac{m-2}{2}}$ is a subharmonic function on the hypersurface. If the hypersurface is compact, by the Hopf theorem, then 
the function $|H|^{\frac{m-2}{2}}$  is constant and hence $H$ is constant, a contradiction.   From which, we get the proposition. 
\end{proof}

We need the following technical proposition.
\begin{proposition}\label{Pro1}
For constants $c>0$, $a_i$, $i=1,2,\ldots,m+1$\;$ (m\neq4)$, and the metric $h_0=\sum\limits_{i=1}^{m}dx_i^2+dz^2$, then a  hypersurface
 $\varphi : (\r^m,g=\varphi^*h) \to (\r^{m+1}, h = \beta^{-2}h_0)$, $\varphi(x_1,\ldots,x_m) = (x_1,\ldots,x_m,\sum\limits_{i=1}^{m}a_ix_i+a_{m+1})$ is proper $f$-biharmonic 
if and only if the function $f=c(1+\sum\limits_{j=1}^{m}a_j^2)^{\frac{4-m}{4}}|-\sum\limits_{i=1}^{m}a_i\beta_i+\beta_z|^{\frac{m-4}{2}}$ is nonconstant, and $\beta$ satisfies the following $PDE$
\begin{equation}\label{pro1}
\begin{array}{lll}
\sum\limits_{i=1}^{m}(\beta\beta_{ii}-m\beta_i^2)+\beta\beta_{zz}-m\beta_z^2
\\+(m-1)(\sum\limits_{i,j=1}^{m}a_ia_j\beta\beta_{ij}-2\sum\limits_{i=1}^{m}a_i\beta\beta_{iz}+\beta\beta_{zz})/(1+\sum\limits_{i=1}^{m}a_i^2)\\
=mH^2-\frac{m-2}{2|H|}\{\beta^2\Delta_{h_0} |H|+(m-2)\beta[H\xi_0(|H|)-{\rm grad}_{h_0}\beta(|H|)]
-\beta^2\xi_0\xi_0(|H|)\}\\
-\frac{(m-2)(m-4)}{4H^2}\beta^2\{|{\rm grad}_{h_0}H|^2-[\xi_0(H)]^2\}
\end{array}
\end{equation}
on the hypersurface, where $\Delta_{h_0}$ is the operator with respect to the
 metric $h_0$, $\xi_0=\frac{-\sum\limits_{i=1}^{m}a_i\frac{\partial}{\partial x_i}+\frac{\partial }{\partial z}}{\sqrt{1+\sum\limits_{j=1}^{m}a_j^2}}$, $H=\xi_0(\beta)$,  $\beta_i=\frac{\partial \beta}{\partial x_i}$, $\beta_{ij}=\frac{\partial^2 \beta}{\partial x_i\partial x_j}$, $\beta_{iz}=\frac{\partial^2 \beta}{\partial x_i\partial z}$ and $\beta_{zz}=\frac{\partial^2 \beta}{\partial z^2}$.
\end{proposition}
\begin{proof}
We can check that the  hypersurface $\varphi : (\r^m,g_0=\varphi^*h_0) \to (\r^{m+1}, h_0= 
 \sum\limits_{i=1}^{m}dx_i^2+dz^2)$, $\varphi(x_1,\ldots,x_m) = (x_1,\ldots,x_m,\sum\limits_{i=1}^{m}a_ix_i+a_{m+1})$ is totally geodesic with  the unit normal vector field $\xi_0=\frac{-\sum\limits_{i=1}^{m}a_i\frac{\partial}{\partial x_i}+\frac{\partial }{\partial z}}{\sqrt{1+\sum\limits_{j=1}^{m}a_j^2}}$. It follows that
 $\varphi : (\r^m,g=\varphi^*h) \to (\r^{m+1}, h = \beta^{-2}(
 \sum\limits_{i=1}^{m}dx_i^2+dz^2))$, $\varphi(x_1,\ldots,x_m) = (x_1,\ldots,x_m,\sum\limits_{i=1}^{m}a_ix_i+a_{m+1})$ is a totally umbilical  hypersurface
 with the unit normal vector field $\xi=\beta\xi_0$ and  the mean curvature $H=\xi_0(\beta)=\xi(\ln\beta)=\frac{-\sum\limits_{i=1}^{m}a_i\beta_i+\beta_z}{\sqrt{1+\sum\limits_{j=1}^{m}a_j^2}}$ (see e.g., \cite{CL,LO}), where $\beta_i=\frac{\partial \beta}{\partial x_i}$.\\
 
With respect to the metric $h$, we apply  Claim II in \cite{Ou2} to compute the Ricci curvature as
\begin{equation}\label{R1}
\begin{array}{lll}
{\rm Ric}^N(\xi,\xi)=\Delta_{h}\ln \beta+(m-1)[{\rm Hess}_{h}(\ln \beta)(\xi,\xi)-
(\xi\ln \beta)^2+|{\rm grad}_{h}\ln \beta|_{h}^2]\\

=\beta\Delta_{h_0}\beta-m|{\rm grad}_{h_0} \beta|^2_{h_0}+(m-1)\beta{\rm Hess}_{h_0} (\beta)(\xi_0,\xi_0)\\

=\sum\limits_{i=1}^{m}(\beta\beta_{ii}-m\beta_i^2)+(\beta\beta_{zz}-m\beta_z^2)
+(m-1)\frac{\sum\limits_{i,j=1}^{m}a_ia_j\beta\beta_{ij}-2\sum\limits_{i=1}^{m}a_i\beta\beta_{iz}+\beta\beta_{zz}}{1+\sum\limits_{i=1}^{m}a_i^2},
\end{array}
\end{equation}
where $\Delta_{h_0}$ and $\Delta_{h}$ denote the operators with respect to the
 metrics $h_0$ and $h$, respectively,  and $\beta_{ij}=\frac{\partial^2 \beta}{\partial x_i\partial x_j}$, $\beta_{iz}=\frac{\partial^2 \beta}{\partial x_i\partial z}$ and $\beta_{zz}=\frac{\partial^2 \beta}{\partial z^2}$.
 
On the other hand,  let us  take an orthonormal frame  $\{e_i,\;\xi\}_{i=1,2,\ldots,m}$ adapted to the hypersurface $(\r^{m},g=\varphi^*h)$,  then
a straightforward computation gives
\begin{equation}\label{R2}
\begin{array}{lll}
\Delta_g |H|=e_ie_i(|H|)-\nabla^M_{e_i}e_i(|H|)\\
=e_ie_i(|H|)+\xi\xi(|H|)-\nabla^{h}_{e_i}e_i(|H|)-\nabla^{h}_{\xi}\xi(|H|)\\
+h(e_i,e_i)\xi(|H|)-\xi\xi(|H|)+\nabla^{h}_{\xi}\xi(|H|)\\
=\Delta_h |H|+mH\xi(|H|)-\xi\xi(|H|)+\nabla^{h}_{\xi}\xi(|H|)\\
=\beta^2\Delta_{h_0} |H|-\beta(m-1)\langle{\rm grad}_{h_0}\beta,{\rm grad}_{h_0}|H|\rangle_{h_0}+mH\beta\xi_0(|H|)\\-\beta\xi_0[\beta\xi_0(|H|)]+\nabla^{h_0}_{\xi}\xi(|H|)-2\beta\xi_0(\ln\beta)\xi(|H|)+\beta^2{\rm grad}_{h_0}\ln\beta(|H|)\\
=\beta^2\Delta_{h_0} |H|-(m-2)\beta\langle{\rm grad}_{h_0}\beta,{\rm grad}_{h_0}|H|\rangle_{h_0}\\+mH\beta\xi_0(|H|)
-\beta^2\xi_0\xi_0(|H|)-2\beta H\xi_0(|H|),
\end{array}
\end{equation}

and 
\begin{equation}\label{R3}
\begin{array}{lll}
|{\rm grad}_{g}(|H|)|^2=\langle{\rm grad}_{h}|H|-\xi(|H|)\xi,{\rm grad}_{h}|H|-\xi(|H|)\xi \rangle\\
=|{\rm grad}_{h}H|^2+[\xi(H)]^2-2\langle{\rm grad}_{h}|H|,\xi(|H|)\xi \rangle
=\beta^2|{\rm grad}_{h_0}H|^2-\beta^2[\xi_0(H)]^2.
\end{array}
\end{equation}
Hence, a direct computation yields

\begin{equation}\label{R4}
\begin{array}{lll}
mH^2-|H|^{\frac{2-m}{2}}\Delta_g(|H|^{\frac{m-2}{2}})\\
=mH^2-\frac{m-2}{2}|H|^{-1}\Delta_g |H|-\frac{(m-2)(m-4)}{4}|H|^{-2}|{\rm grad}_g (|H|)|^2\\
=mH^2-\frac{m-2}{2|H|}\{\beta^2\Delta_{h_0} |H|+(m-2)\beta[H\xi_0(|H|)-{\rm grad}_{h_0}\beta(|H|)]
-\beta^2\xi_0\xi_0(|H|)\}\\
-\frac{(m-2)(m-4)}{4H^2}\beta^2\{|{\rm grad}_{h_0}H|^2-[\xi_0(H)]^2\}.
\end{array}
\end{equation}
Since the hypersurface is totally umbilical, then  $({\rm Ric}^N(\xi))^{\top}=(m-1){\rm grad_g}H$ holds naturally.  Using Corollary \ref{FS2} with (\ref{R1}) and (\ref{R4}), we obtain the proposition. 
\end{proof}

If $\beta=\beta(z)$ depends  on only the variable $z$ in Proposition \ref{Pro1}, then  the following theorem  provides a main technique
to construct $f$-biharmonic hpersurfaces.
\begin{theorem}\label{PQ0}
The  hypersurface $\varphi : (\r^m,g=\varphi^*h) \to (\r^{m+1}, h = \beta^{-2}(z)(
 \sum\limits_{i=1}^{m}dx_i^2+dz^2))$, $\varphi(x_1,\ldots,x_m) = (x_1,\ldots,x_m,\sum\limits_{i=1}^{m}a_ix_i+a_{m+1})$  is proper $f$-biharmonic
if and only if the function $f=ck_m^{\frac{m-4}{2}}|\beta'(z)|^{\frac{m-4}{2}}$ is nonconstant  and $\beta(z)$ satisfies  the following $ODE$
 \begin{equation}\label{PQ1}
\begin{array}{lll}
m(1+k_m^2)\beta'^4+\frac{(m^2-2m+2)(1-k_m^2)-2m}{2}\beta\beta'^2\beta''\\-\frac{(m-2)(1-k_m^2)}{2}\beta^2\beta'\beta'''
-\frac{(m-2)(m-4)(1-k_m^2)}{4}\beta^2\beta''^2=0
\end{array}
\end{equation}
 on the hypersurface, where $c>0$, $a_1, a_2,\ldots,a_m$ and $a_{m+1}$ are constants, and $k_m=1/\sqrt{1+\sum\limits_{i=1}^{m}a_i^2}$.
\end{theorem}
\begin{proof}
We know from Proposition \ref{Pro1} that the hypersurface $\varphi : (\r^m,g=\varphi^*h) \to (\r^{m+1}, h=
 \beta^{-2}(z)(\sum\limits_{i=1}^{m}dx_i^2+dz^2))$, $\varphi(x_1,\ldots,x_m) = (x_1,\ldots,x_m,\sum\limits_{i=1}^{m}a_ix_i+a_{m+1})$   is totally umbilical with the unit normal vector field $\xi=\beta\frac{-\sum\limits_{i=1}^{m}a_i\frac{\partial}{\partial x_i}+\frac{\partial }{\partial z}}{\sqrt{1+\sum\limits_{j=1}^{m}a_j^2}}$ and the mean curvature $H=\xi(\ln\beta)=\beta'(z)/\sqrt{1+\sum\limits_{j=1}^{m}a_j^2}$.
  
For $\xi_0=(-\sum\limits_{i=1}^{m}a_i\frac{\partial}{\partial x_i}+\frac{\partial }{\partial z})/(\sqrt{1+\sum\limits_{j=1}^{m}a_j^2})$ and $h_0=\sum\limits_{i=1}^{m}dx_i^2+dz^2$,  a straightforward computation gives
\begin{equation}\label{pq1}
\begin{array}{lll}
=mH^2-\frac{m-2}{2|H|}\{\beta^2\Delta_{h_0} |H|+(m-2)\beta[H\xi_0(|H|)-{\rm grad}_{h_0}\beta(|H|)]
-\beta^2\xi_0\xi_0(|H|)\}\\
-\frac{(m-2)(m-4)}{4H^2}\beta^2\{|{\rm grad}_{h_0}H|^2-[\xi_0(H)]^2\}\\
=mk_m^2\beta'^2-\frac{m-2}{2k_m\beta'}[k_m(1-k_m^2)\beta^2\beta'''+k_m(1-k_m^2)(2-m)\beta\beta'\beta'']\\
-\frac{(m-2)(m-4)}{4k_m^2\beta'^2}[k_m^2(1-k_m^2)\beta^2\beta''^2]\\
=mk_m^2\beta'^2-\frac{(m-2)(1-k_m^2)}{2} \frac{\beta^2\beta'''+(2-m)\beta\beta'\beta''}{\beta'}
-\frac{(m-2)(m-4)(1-k_m^2)}{4}\frac{\beta^2\beta''^2}{\beta'^2}.
\end{array}
\end{equation}
Note that $\beta_{i}=\frac{\partial \beta}{\partial x_i}=0, \beta_{ij}=\frac{\partial^2 \beta}{\partial x_i\partial x_j}=0$ and $\beta_{iz}=0$ for $i=1,2,\ldots,m$.  Substituting this and (\ref{pq1}) into  (\ref{pro1}) and simplifying the resulting equation yields (\ref{PQ1}). Combining these and using  Proposition \ref{Pro1},  we get the theorem.
\end{proof}
As an immediate consequence of Theorem \ref{PQ0}, we have
\begin{corollary}\label{pq0}
For $\beta=\beta(z)$, then a  surface $\varphi : (\r^2,g=\varphi^*h) \to (\r^{3}, h = \beta^{-2}(z)(
 \sum\limits_{i=1}^{2}dx_i^2+dz^2))$, $\varphi(x_1,x_2) = (x_1,x_2,a_1x_1+a_2x_2+a_{3})$ into a conformally flat space is proper $f$-biharmonic
if and only if the function $f=\frac{c\sqrt{1+a_1^2+a_2^2}}{|\beta'(z)|}$ is nonconstant  and $\beta$ satisfies the following $ODE$
 \begin{equation}\label{pq01}
\begin{array}{lll}
\beta\beta''-2\beta'^2=0
\end{array}
\end{equation}
on the surface,  where $c>0$ and  $a_{i}$, i=1,2,3,  are constants.
\end{corollary}

\begin{remark}\label{re1}
We can check that there is at least  one nonzero constant $a_i$ for  $i=1,\ldots,m$ in Theorem \ref{PQ0} and Corollary \ref{pq0}.
\end{remark}

If $a_i=0$  in Proposition \ref{Pro1} for $i=1,2,\ldots,m$, then  the following  theorem  provides another  main technique
to construct $f$-biharmonic hpersurfaces.
\begin{theorem}\label{PC1}
A   hypersurface $\varphi : (\r^m,g=\varphi^*h) \to (\r^{m+1}, h = \beta^{-2}(
 \sum\limits_{i=1}^{m}dx_i^2+dz^2))$, $\varphi(x_1,\ldots,x_m) = (x_1,\ldots,x_m,a_{m+1})$ into a conformally flat space is proper $f$-biharmonic
if and only if the function $f=c|\beta_z|^{\frac{m-4}{2}}$ is nonconstant  and $\beta$ satisfies the following $PDE$
 \begin{equation}\label{pc1}
\begin{array}{lll}
\sum\limits_{i=1}^m(\beta\beta_{ii}-m\beta_i^2)+m(\beta\beta_{zz}-2\beta_z^2)\\+\frac{(m-2)\sum\limits_{i=1}^m[\beta^2\beta_{iiz}-(m-2)\beta\beta_i\beta_{iz}]}{2\beta_z}
+\frac{(m-2)(m-4)\sum\limits_{i=1}^m\beta^2\beta^2_{iz}}{4\beta^2_z}
=0
\end{array}
\end{equation}
on the hypersurface,  where $c>0$ and  $a_{m+1}$  are constants.
\end{theorem}
\begin{proof}
If $a_i=0$  in Proposition \ref{Pro1} for $i=1,2,\ldots,m$, then we see that $\xi_0=\frac{\partial }{\partial z}$, $H=\xi_0(\beta)=\beta_z$. Substituting these into (\ref{pro1}) and simplifying the resulting equation we obtain (\ref{pc1}). Clearly,  $f=c|\beta_z|^{\frac{m-4}{2}}$. Using these and applying Proposition \ref{Pro1}, the theorem follows.
\end{proof}
As an immediate consequence of Theorem \ref{PC1}, we have
\begin{corollary}\label{tr3}
A  surface $\varphi : (\r^2,g=\varphi^*h) \to (\r^{3}, h = \beta^{-2}(
 \sum\limits_{i=1}^{2}dx_i^2+dz^2))$, $\varphi(x_1,x_2) = (x_1,x_2,a_{3})$ into a conformally flat space is proper $f$-biharmonic
if and only if the function $f=\frac{c}{|\beta_z|}$ is nonconstant  and $\beta$ satisfies the following $PDE$
 \begin{equation}\label{ppc1}
\begin{array}{lll}
\sum\limits_{i=1}^2(\beta\beta_{ii}-2\beta_i^2)+2(\beta\beta_{zz}-2\beta_z^2)=0
\end{array}
\end{equation}
on the surface,  where $c>0$ and  $a_{3}$  are constants.
\end{corollary}

\section{ $f$-Biharmonic surfaces and biharmonic conformal immersions}
In this section,  we  use   $f$-biharmonic surfaces to construct  biharmonic conformal immersions  whilst
we construct $f$-biharmonic surfaces  by using given  conformal immersions of surfaces.
We also obtain  a classification result on $f$-biharmonic surfaces in  3-space forms.

The interesting link among biharmonicity, $f$-biharmonicity and conformality  from a 2 manifold can be stated as
\begin{lemma}( \cite{Ou3})\label{cf1}
A map $\phi : (M^2, g) \to (N^n, h)$ is an $f$-biharmonic map if and only if $\phi : (M^2, f ^{-1}g) \to (N^n, h)$ is a biharmonic map. In particular, a surface (i.e., an isometric immersion) $ \phi : (M^2, g = \phi^{*}h)\to(N^n, h)$ is an $f$-biharmonic surface if and only if the conformal immersion $\phi : (M^2, f ^{-1}g)\to (N^n, h)$ is a 
biharmonic map with the conformal factor $\lambda=f^{\frac{1}{2}}$.
\end{lemma}

Applying Corollary \ref{pq0}, we have the following proposition which  privodes proper $f$-biharmonic surfaces in a conformally flat 3-space.
\begin{proposition}\label{tr1}
Let $a_1, a_2,a_3$, $c_1, c_2$ and $c$ be positive constants, let $$\r^{3}_{+}=\{(x_1,x_2,z):z>0\}$$
denote the upper half-space, and let $\beta:\r^{3}_{+}\to (0,+\infty)$ with $\beta=\frac{1}{c_1z+c_2}$.
Then, for  the  function $f$ from the
family $f=\frac{c\sqrt{1+a_1^2+a_2^2}(c_1a_1x+c_1a_2y+c_1a_3+c_2)^2}{c_1}$, the  surface $$\varphi : (\r^2,g=\varphi^*h) \to \left(\r^{3}_{+}, h = \beta^{-2}(z)\left[
 \sum\limits_{i=1}^{2}dx_i^2+dz^2\right]\right)$$ with $\varphi(x_1,x_2) = (x_1,x_2,\sum\limits_{i=1}^{2}a_ix_i+a_{3})$  is proper $f$-biharmonic.
\end{proposition}
\begin{proof}
 If a surface $$\varphi : (\r^2,g=\varphi^*h) \to \left(\r^{3}_{+}, h = \beta^{-2}(z)\left[
 \sum\limits_{i=1}^{2}dx_i^2+dz^2\right]\right)$$ with $\varphi(x_1,x_2) = (x_1,x_2,\sum\limits_{i=1}^{2}a_ix_i+a_{3})$  is proper $f$-biharmonic,  by Corollary \ref{pq0}, then $\beta(z)$ satisfies the  equation $\beta\beta''(z)-2\beta'^2(z)=0$ on the surface,
which is solved by $\beta=\frac{1}{c_1z+c_2}$, where $c_1$ and $c_2$ are positive constants. 
From this, we get the proposition.
\end{proof}

By Corollary \ref{tr3}, we can also obtain proper $f$-biharmonic surfaces.
\begin{proposition}\label{tr4}
Let $c_1, c_2, c_3, c_4$, $a_3$ and $c$ be positive constants, and let $$\r^{3}_{+}=\{(x_1,x_2,z):z>0\}$$
denote the upper half-space. For $f=\frac{c(c_3a_3+c_4)^2(c_1x_i+c_2)}{c_3}$ and $\beta=\frac{1}{(c_1x_i+c_2)(c_3z+c_4)}$, $i=1,2$,
then  the surface $\varphi : (\r^2,\varphi^*h) \to (\r^{3}_{+}, h = \beta^{-2}(
 \sum\limits_{i=1}^{2}dx_i^2+dz^2))$, $\varphi(x_1,x_2) = (x_1,x_2,a_{3})$ into a conformally flat space is proper $f$-biharmonic.
\end{proposition}

\begin{proof}
We will look for  special solutions of  (\ref{ppc1}) that have the form $\beta=p(x_i)q(z)$ being nonconstant for $i=1,2$.  Substituting this into 
(\ref{ppc1}) and simplifying the resulting equation yields
\begin{equation}\label{pOP1}
\begin{array}{lll}
q^2(z)[ p(x_i)p''(x_i)-2p'^2(x_i)]+2p^2(x_i)[ q(z)q''(z)-2q'^2(z)]=0.
\end{array}
\end{equation}
Since $\beta=p(x_i)q(z)$ is nonconstant on the target manifold, then (\ref{pOP1}) turns into
\begin{equation}\label{pOP2}
\begin{array}{lll}
 p(x_i)p''(x_i)-2p'^2(x_i)=0,\;q(z)q''(z)-2q'^2(z)=0.
\end{array}
\end{equation}
Solving  (\ref{pOP2}), we have
\begin{equation}\label{OP3}
\begin{array}{lll}
p(x_i)=(c_1x_i+c_2)^{-1},\;\;q(z)=(c_3z+c_4)^{-1},
\end{array}
\end{equation} where $c_j$ are constants for $j=1,2,3,4$.

Combining these and using Corollary \ref{tr3}, we obtain the proposition.
\end{proof}

By Lemma \ref{cf1} and Proposition \ref{tr1}, we can construct proper biharmonic conformal immersions into a conformally flat space.
\begin{corollary}\label{tr2}
Adopting the same notations as in Proposition \ref{tr1},
for the positive function  $f=\frac{c\sqrt{1+a_1^2+a_2^2}(c_1a_1x+c_1a_2y+c_1a_3+c_2)^2}{c_1}$, then the  conformal immersion $$\varphi : (\r^2,f^{-1}\varphi^*h) \to \left(\r^{3}, h = (c_1z+c_2)^2\left[
 \sum\limits_{i=1}^{2}dx_i^2+dz^2\right]\right)$$ with $\varphi(x_1,x_2) = (x_1,x_2,\sum\limits_{i=1}^{2}a_ix_i+a_{3})$  is proper biharmonic.

\end{corollary}

Using Lemma \ref{cf1} and Proposition \ref{tr4}, one can also give  proper biharmonic conformal immersions into a conformally flat space.
\begin{corollary}\label{tr5}
Adopting the same notations as in Proposition \ref{tr4}, for $f=\frac{c(c_3a_3+c_4)^2(c_1x_i+c_2)}{c_3}$ and $\beta=\frac{1}{(c_1x_i+c_2)(c_3z+c_4)}$,
then  conformal immersion $$\varphi : (\r^2,f^{-1}\varphi^*h) \to (\r^{3}, h = \beta^{-2}(
 \sum\limits_{i=1}^{2}dx_i^2+dz^2))$$ with $\varphi(x_1,x_2) = (x_1,x_2,a_{3})$ into a conformally flat space is proper biharmonic.
\end{corollary}
On the other hand, we can use Proposition \ref{cf1} to construct  proper $f$-biharmonic surfaces. Note that Statement (iii) of Theorem 2.6 in \cite{WC} can be stated as
\begin{proposition}\label{tr6}
 Let   $\beta=\frac{2r^3}{-2rz-2(r^2+z^2)\arctan\frac{z}{r}+k\;r^3(r^2+z^2)}$, where
 $r=\sqrt{1+x^2+y^2}$ and $k\geq6$ is a constant. 
For $f=\frac{k^2(1+x^2+y^2)^2}{4}$, then the conformal immersion $\varphi : (S^2\backslash\{N\},g_0=\frac{4(dx^2+dy^2)}{(1+x^2+y^2)^2}) \to(S^3\backslash\{N'\}, h=\beta^{-2}\;\frac{4(dx^2+dy^2+dz^2)}{(1+x^2+y^2+z^2)^2})$  with $\varphi(x, y) = (x,y,0)$  into the conformal 3-sphere is proper biharmonic. 
\end{proposition}

Using Lemma \ref{cf1} and Proposition \ref{tr6}, we obtain a family of proper $f$-biharmonic surfaces into the conformal 3-spheres.

\begin{corollary}\label{tr7}
 Let   $\beta=\frac{2r^3}{-2rz-2(r^2+z^2)\arctan\frac{z}{r}+k\;r^3(r^2+z^2)}$, where
 $r=\sqrt{1+x^2+y^2}$ and  $k\geq6$ is a constant. 
For $f=\frac{k^2(1+x^2+y^2)^2}{4}$, then the surface $\varphi : (\r^2,\varphi^{*}h=k^2(dx^2+dy^2)) \to(S^3\backslash\{N'\}, h=\beta^{-2}\;\frac{4(dx^2+dy^2+dz^2)}{(1+x^2+y^2+z^2)^2})$  with $\varphi(x, y) = (x,y,0)$  into the conformal 3-sphere is proper $f$-biharmonic. 
\end{corollary}

At the end of this section, we are ready to give a classification result on $f$-biharmonic surfaces in  3-space forms.

\begin{theorem}\label{cs}
A nonzero constant mean curvature surface $\varphi:(M^2,g)\to
N^3(C)$  into a 3-space form of constant sectional curvature $C$ 
 is proper $f$-biharmonic if and only if $C=0$, and 
it is a part of  a circular cylinder of radius $\frac{1}{2|H|}$ in $\r^3$, where  $H$ denotes the mean curvature of the surface.
\end{theorem}
\begin{proof}
 If the surface of  nonzero constant mean curvature $H$  into a 3-space form of constant sectional curvature $C$ is $f$-biharmonic, by
 Theorem \ref{ff0}, then  we have
\begin{equation}\label{ccs1}
\Delta f=(|A|^2-2C)f,\;
A({\rm grad}f)=0.
\end{equation}

To solve the equation (\ref{ccs1}),  we take an orthonormal frame $\{e_1, e_2, \xi\}$ on $ N^3(C)$ adapted to the surface $M^2$ such that $A_\xi(e_i)=\lambda_i e_i$, where $\lambda_i$ are the principal curvatures in the directions $e_i$,  respectively, then the 2nd equation of (\ref{ccs1}) becomes
\begin{equation}\label{ccs4}
\lambda_1e_1(f)=0,\;
\lambda_2e_2(f)=0.
\end{equation}
Note that $f$ is nonconstant since the map $\varphi$ is proper $f$-biharmonic. By the assumption that $H\neq0$, then $\lambda_1^2+\lambda_2^2\neq0$.
Without loss of generality, we may assume that $\lambda_1\neq0$.  Hence,  we use (\ref{ccs4}) to get
\begin{equation}\label{ccs7}
\lambda_2=0\;\;{\rm and}\;
e_1(f)=0.
\end{equation}
On the other hand, by the Gauss equation we have
\begin{equation}\label{ccs8}
 R^M(e_1,e_2,e_1,e_2)=R^N(e_1,e_2,e_1,e_2)+\lambda_1\lambda_2,
\end{equation}
which, together with $\lambda_2=0$,  implies that the surface $M^2$ has constant Gauss curvature
$ K^{M^2}=R^M(e_1,e_2,e_1,e_2)=R^N(e_1,e_2,e_1,e_2)=C$. 

Using the Coddazzi equation gives
\begin{equation}\label{css10}
e_1(\lambda_2)=(\lambda_1-\lambda_2)\langle\nabla_{e_2}e_1,e_2\rangle,\;
e_2(\lambda_1)=(\lambda_2-\lambda_1)\langle\nabla_{e_1}e_2,e_1\rangle.
\end{equation}
This, together with $\lambda_2=0$ and  nonzero constant $\lambda_1=2H$, means that
$\langle\nabla_{e_1}e_1,e_2\rangle=\langle\nabla_{e_2}e_2,e_1\rangle=0.$
Hence, it follows that $C=K^{M^2}=R^N(e_1,e_2,e_1,e_2)=e_2\langle\nabla_{e_1}e_1,e_2\rangle+e_1\langle\nabla_{e_2}e_2,e_1\rangle-\langle\nabla_{e_1}e_1,e_2\rangle^2-\langle\nabla_{e_2}e_2,e_1\rangle^2=0$,
which implies that the target space $N^3(C)$ has to be a Euclidean space $\r^3$. Combining these and applying Theorem  2.10 in \cite{Ou4}, it follows that the surface $(M^2,g)$ in a Euclidean space $\mathbb{R}^3$ is proper $f$-biharmonic if and only if it is apart of  a circular cylinder of radius $\frac{1}{2|H|}$ in $\r^3$.

This completes the proof of the theorem.
\end{proof}
Applying Lemma \ref{cf1} and Theorem \ref{cs}, we immediately have the following

\begin{corollary}\label{css}
Let  $\varphi:(M^2,g)\to N^3(C)$  be a  surface with nonzero constant mean $H$ into a 3-space form of constant sectional curvature $C$.
Then,  the conformal immersion $\varphi:(M^2,f^{-1}g)\to N^3(C)$  with  nonconstant conformal factor $f^{1/2}$ is proper biharmonic if and only if $C=0$, and 
the surface  is a part of  a circular cylinder of radius $\frac{1}{2|H|}$ in $\r^3$.
\end{corollary}

\begin{remark}\label{re3}
The author in \cite{Ou4} proved  that no part of the standard sphere $S^2$ can be biharmonically conformally immersed into $\r^3$ and a nonzero constant mean curvature surface can be only  biharmonically conformally immersed into $\r^3$ if and only if it is a part of  a circular cylinder with a different way. Our Corollary \ref{css} recovers the above results. 
\end{remark}
\begin{remark}\label{re4}
By Corollary \ref{FS2},  a totally umbilical proper $f$-biharmonic surface $M^2 \to N^3 $  with the unit normal vector field
$\xi$  has the Ricci curvature ${\rm Ric}^N(\xi,\xi)=2H^2>0$,  which implies that  $N^3$  can not  be a nonpositively curved manifold.
\end{remark}

\section{ Proper $f$-biharmonic hypersurfaces $M^m$ with $m\geq3$}
 We earlier mentioned that the generalized Chen’s conjecture
has been proved to be false  by constructing  proper
bihmarmonic hyperplanes in a five-dimensional conformally flat space of negative sectional curvature in \cite{Ou2,LiO} where
the authors  used this and the biharmonicity of the product maps to produce many examples of proper biharmonic submanifolds
in nonpositively curved Riemannian manifolds.  Following the idea in \cite{Ou2,LiO},  we investigate $f$-biharmonicity of  totally umbilical  hyperplanes
into  the conformally flat $n$-spaces in this section. By seeking  special  solutions of (\ref{PQ1}) and (\ref{pc1}),  we construct the examples of totally umbilical proper $f$-biharmonic hypersurfaces $M^m$ with $m\geq3$ and $m\neq4$ into the conformally flat space with strictly negative sectional curvature. Furthermore,  by using such hypersurfaces followed by a totally geodesic embedding, we provide  proper $f$-biharmonic submanifolds in nonpositively curved manifolds.\\

We give the following proposition which  produces proper $f$-biharmonic hypersurfaces in a conformally flat space with negative sectional curvature.
\begin{proposition}\label{pqe1}
Let $c>0$,  $a_1$, $a_2$, $\ldots,a_{m}$ and  $a_{m+1}$ be  constants with $\sum\limits_{i=1}^{m}a_i^2=m$, let $$\r^{m+1}_{+}=\{(x_1,x_2,\ldots,x_m,z):z>0\}$$
denote the upper half-space. Then, \\
(i) for $\beta=z^{-1}$ and $f=c(m+1)^{\frac{4-m}{4}}[\sum\limits_{i=1}^{m}a_ix_i+a_{m+1}]^{4-m}$ with $m\neq4$, the  hypersurface $$\varphi : (\r^m,g=\varphi^*h) \to \left(\r^{m+1}_{+}, h = \beta^{-2}(z)\left[
 \sum\limits_{i=1}^{m}dx_i^2+dz^2\right]\right)$$ with $\varphi(x_1,\ldots,x_m) = (x_1,\ldots,x_m,\sum\limits_{i=1}^{m}a_ix_i+a_{m+1})$  is proper $f$-biharmonic;\\
 (ii)  for $\beta=z^{\frac{m^2-2m}{m^2+4}}$ and $f=c[\frac{(m^2-2m)^2}{(m+1)(m^2+4)^2}]^{\frac{m-4}{4}}[\sum\limits_{i=1}^{m}a_ix_i+a_{m+1}]^{\frac{(4-m)(m+2)}{m^2+4}}$ with $m\geq3$ and $m\neq4$, the  hypersurface $$\varphi : (\r^m,g=\varphi^*h) \to \left(\r^{m+1}_{+}, h = \beta^{-2}(z)\left[
 \sum\limits_{i=1}^{m}dx_i^2+dz^2\right]\right)$$ with $\varphi(x_1,\ldots,x_m) = (x_1,\ldots,x_m,\sum\limits_{i=1}^{m}a_ix_i+a_{m+1})$  is proper $f$-biharmonic.

\end{proposition}
\begin{proof}
 We  look for the special  solution of (\ref{PQ1}) of the form $\beta=z^t$. We also suppose that $k_m=1/\sqrt{1+\sum\limits_{i=1}^{m}a_i^2}=\sqrt{\frac{1}{m+1}}$, i.e., $\sum\limits_{i=1}^{m}a_i^2=m$. Substituting these into 
 (\ref{PQ1}) and simplifying the resulting equation yields
 
\begin{equation}\label{pqe3}
\begin{array}{lll}
\frac{[(m^2+4)t^2+(2m+4)t+2m-m^2]mt^2z^{4t^2-4}}{4(m+1)}=0,
\end{array}
\end{equation}
which implies
\begin{equation}\label{pqe4}
\begin{array}{lll}
(m^2+4)t^2+(2m+4)t+2m-m^2=0.
\end{array}
\end{equation}
Solving (\ref{pqe4}) we have either $t=-1$ or $t=\frac{m^2-2m}{m^2+4}$. Therefore, we obtain the special  solutions of (\ref{PQ1}) as  $\beta=z^{-1}$ or $\beta=z^{\frac{m^2-2m}{m^2+4}}$ with $m\geq3$. Hence, a further computation, we have $f=ck_m^{\frac{m-4}{2}}|\beta'(z)|^{\frac{m-4}{2}}=c(m+1)^{\frac{4-m}{4}}|z|^{4-m}$ or  $f=ck_m^{\frac{m-4}{2}}|\beta'(z)|^{\frac{m-4}{2}}=c[\frac{(m^2-2m)^2}{(m+1)(m^2+4)^2}]^{\frac{m-4}{4}}|z|^{\frac{(4-m)(m+2)}{m^2+4}}$, where $z=\sum\limits_{i=1}^{m}a_ix_i+a_{m+1}$. Combining these and using Theorem \ref{PQ0}, the Proposition follows.

\end{proof}

 We can use the following proposition to privode  proper $f$-biharmonic hypersurfaces in a conformally flat space of negative sectional curvature.
\begin{proposition}\label{PC2}
Let $c$,  $C$ and $a_{m+1}$ be positive constants, and let $\r^{m+1}_{*}=\{(x_1,\ldots,x_m,z)\in\r^{m+1}:x_1,\ldots,x_m,z>0\}$. Then, \\
(i) for $\beta=(\sum\limits_{i=1}^mx_i+a_{m+1}+C)^{-1}$  and $f=c(\sum\limits_{i=1}^mx_i+a_{m+1}+C)^{4-m}$ with $m\neq4$, the   hypersurface $$\varphi : (\r^m,g=\varphi^*h) \to \left(\r^{m+1}_{*}, h =\beta^{-2}(x_1,\ldots,x_m,z) \left[
 \sum\limits_{j=1}^{m}dx_j^2+dz^2\right]\right)$$ with $\varphi(x_1,\ldots,x_m) = (x_1,\ldots,x_m,a_{m+1})$  is proper $f$-biharmonic;

(ii) for $\beta=(\sum\limits_{i=1}^mx_i+z+C)^{\frac{m^2-2m}{m^2+4}}$ and $f=c(\frac{m^2-2m}{m^2+4})^{\frac{m-4}{2}}(\sum\limits_{i=1}^mx_i+a_{m+1}+C)^{\frac{(4-m)(m+2)}{m^2+4}}$ with $m\geq3$, $m\neq4$,
the   hypersurface $$\varphi : (\r^m,g=\varphi^*h) \to \left(\r^{m+1}_{*}, h =\beta^{-2}(x_1,\ldots,x_m,z) \left[
 \sum\limits_{j=1}^{m}dx_j^2+dz^2\right]\right)$$ with $\varphi(x_1,\ldots,x_m) = (x_1,\ldots,x_m,a_{m+1})$  with $\varphi(x_1,\ldots,x_m) = (x_1,\ldots,x_8,a_{9})$  is proper $f$-biharmonic.
\end{proposition}
\begin{proof}
We seek  special solutions of  (\ref{pc1}) that have the form $\beta=(\sum\limits_{i=1}^mx_i+z+C)^t$ being nonconstant, where $t\in\r$ and $C$ is a constant.  Substituting this into 
(\ref{pc1}) and simplifying the resulting equation yields
\begin{equation}\label{OP2}
\begin{array}{lll}
m\{(m^2+4)t^2+(2m+4)t+2m-m^2\}(\sum\limits_{i=1}^mx_i+z+C)^{2t-2}=0.
\end{array}
\end{equation}
This implies $(m^2+4)t^2+(2m+4)t+2m-m^2=0$ and hence either $t=-1$ or $t=\frac{m^2-2m}{m^2+4}$.
Then,  a simple computation, we have $f=c|\beta_z|^{\frac{m-4}{2}}|_{z=a_{m+1}}=c|\sum\limits_{i=1}^mx_i+a_{m+1}+C|^{4-m}$ or  $f=c|\beta_z|^{\frac{m-4}{2}}|_{z=a_{m+1}}=c(\frac{m^2-2m}{m^2+4})^{\frac{m-4}{2}}|\sum\limits_{i=1}^mx_i+a_{m+1}+C|^{\frac{(4-m)(m+2)}{m^2+4}}$.
Combining these and using Theorem \ref{PC1}, we obtain the proposition.
\end{proof}

The following two model spaces are of negative sectional curvatures.
\begin{lemma}\label{sec1}
Both the conformally flat spaces $(\r^{m+1}_{+}, z^{-2\frac{m^2-2m}{m^2+4}}[
 \sum\limits_{i=1}^{m}dx_i^2+dz^2])$  and $(\r^{m+1}_{*}, (\sum\limits_{i=1}^mx_i+z+C)^{-2\frac{m^2-2m}{m^2+4}}[
 \sum\limits_{i=1}^{m}dx_i^2+dz^2])$ with $m\geq3$ have negative sectional curvatures. 
\end{lemma}
\begin{proof}
Let $\beta=(A\sum\limits_{i=1}^mx_i+z+C)^{t}$ with  $t=\frac{m^2-2m}{m^2+4}$, where $A=0$ or $A=1$ and $C\geq0$ is a  constant. Note that if $m\geq3$, then $0<t=\frac{m^2-2m}{m^2+4}<1$. Choose the  orthonormal frame  $e_i=\beta \frac{\partial}{\partial x_i}$ for $i=1,2,\ldots,m$, and $e_{m+1}=\beta \frac{\partial}{\partial z}$ on
$(\r^{m+1}, h = (A\sum\limits_{i=1}^mx_i+z+C)^{-2t}[
 \sum\limits_{i=1}^{m}dx_i^2+dz^2])$  with $m\geq3$. Let $X=\sum\limits_{i=1}^{m+1}X_ie_i$ and $Y=\sum\limits_{i=1}^{m+1}Y_ie_i$ be an orthonormal basis on a plane section at any point.
Then, one finds the sectional curvature of the conformally flat space to be (see e.g.,\cite{Ou2,LiO}) 
\begin{equation}\label{se1}
\begin{array}{lll}
K(X,Y)=XX(\ln\beta)+YY(\ln\beta)-\nabla_{X}X(\ln\beta)-\nabla_{Y}Y(\ln\beta)\\+(|\nabla\ln\beta|^2-(X\ln\beta)^2-(Y\ln\beta)^2)\\
=\sum\limits_{i,j=1}^{m+1}(X_iX_j+Y_iY_j)\beta\beta_{ij}-\sum\limits_{i=1}^{m+1}\beta_i^2\\
=\{A^2[(\sum\limits_{i=1}^{m}X_i)^2+(\sum\limits_{i=1}^{m}Y_i)^2]t(t-1) +2A(\sum\limits_{i=1}^{m}X_iX_{m+1}+Y_iY_{m+1})t(t-1)\\-A^2mt^2 +(X_{m+1}^2+Y_{m+1}^2)t(t-1)-t^2\}(A\sum\limits_{i=1}^mx_i+z+C)^{2t-2},
\end{array}
\end{equation}
where we denote by $x_{m+1}=z$, $\beta_i=\frac{\partial \beta}{\partial_{x_i}}$ and $ \beta_{ij}=\frac{\partial \beta^2}{\partial_{x_j}\partial_{x_i}}$, $i,j=1,2,\ldots,m+1$.\\
It follows that 
\begin{equation}\label{se2}
K(X,Y)=\begin{cases}
\{(X_{m+1}^2+Y_{m+1}^2)t(t-1)-t^2\}(z+C)^{2t-2}, \;\;\;\;{\rm for}\;A=0,\\

\{[(\sum\limits_{i=1}^{m+1}X_i)^2+(\sum\limits_{i=1}^{m+1}Y_i)^2]t(t-1) \\-(m+1)t^2\}(\sum\limits_{i=1}^mx_i+z+C)^{2t-2},\;\;{\rm for}\;A=1.
\end{cases}
\end{equation}
We conclude from (\ref{se2}) that if $z+C>0$ and $\sum\limits_{i=1}^mx_i+z+C>0$, then $K(X,Y)<0$ since $0<t<1$. Combining these,  we get the proposition.
\end{proof}
To end this section, we investigate $f$-biharmonicity of  submanifolds
in a conformally flat space with nonpositve sectional curvature.
\begin{corollary}\label{qlyx}
(i) Any  totally umbilical $f$-biharmonic surface of $N^3$ with nonpositve sectional curvature
is minimal. The surface  is actually totally geodesic.\\
(ii) There exist  infinitely many proper $f$-biharmonic  $m$-dimensional submanifolds with $m\geq3$ and $m\neq4$
into  the  nonpositvely curved manifolds. 

\end{corollary}

\begin{proof}
For $m=2$, we can conclude from  Theorem \ref{FS1}  that if a  totally umbilical surface  $\varphi:M^2 \to N^3 $   with the unit normal vector field
$\xi$  is $f$-biharmonic, then it is  a totally geodesic surface,  or it has the Ricci curvature ${\rm Ric}^N(\xi,\xi)=2H^2>0$.  Therefore, we see that any  totally umbilical $f$-biharmonic surface of $N^3$ with strictly negative sectional curvature has to be totally geodesic and hence it is minimal.  From this, we obtain the first statement.\\

For $m\geq3$ and $m\neq4$,  we conclude from Statement (ii) of Proposition \ref{pqe1}, Statement (ii) of Proposition \ref{PC2} and Lemma \ref{sec1} that there are many examples of  totally umbilical  proper $f$-biharmonic  hypersurfaces into   conformally flat $n$-spaces  with strictly negative sectional curvatures.  For instance, let $$\varphi : (\r^m,g=\varphi^*h) \to \left(\r^{m+1}_{+}, h = \beta^{-2}(z)\left[
 \sum\limits_{i=1}^{m}dx_i^2+dz^2\right]\right)$$ with $\varphi(x_1,\ldots,x_m) = (x_1,\ldots,x_m,\sum\limits_{i=1}^{m}a_ix_i+a_{m+1})$ be one of the proper $f$-biharmonic
hypersurfaces given in Statement (ii) of Proposition \ref{pqe1}; let $\psi: (\r^{m+1}_{+}, h ) \to  (\r^{m+1}_{+}\times\r^k, h +h_1)) $ with $\psi(x_1,\ldots,x_m,z) = (x_1,\ldots,x_m,z,0,\ldots,0)$ be the totally geodesic embedding, where $k$ is a positive integer and $(\r^k, h_1)$ denotes a  Euclidean
space. It follows from Remark 3 in \cite{WQ} that the submanifold $\phi=\psi\circ\varphi: (\r^m,\varphi^*(h+h_1)) \to (\r^{m+1}_{+}\times\r^k, h+h_1)$ with $\phi(x_1,\ldots,x_m) = (x_1,\ldots,x_m,\sum\limits_{i=1}^{m}a_ix_i+a_{m+1},0,\ldots.0)$ is a proper $f$-biharmonic submanifold. By Lemma \ref{sec1},  the conformally flat space $(\r^{m+1}_{+}, z^{-2\frac{m^2-2m}{m^2+4}}[
 \sum\limits_{i=1}^{m}dx_i^2+dz^2])$ has negative sectional curvature. Note that the Euclidean space $(\r^k
, h_1)$ has zero curvature.  From these, we see that  their product
$(\r^{m+1}_{+}\times\r^k, h+h_1)$ has non-positive sectional curvature.
Combining these, the second statement follows.
\end{proof}

Statements and Declarations:\\
1. Funding: Ze-Ping Wang was supported by the Natural Science Foundation of China (No. 11861022) and by the Scientific and Technological Project in Guizhou Province ( Grant nos. Qiankehe Platform Talents [2018]5769-04, LH[2017]7342).\\
2. Conflict of interest: The authors declare that they have no conflict of interest.\\
3. Data Availability Statement: The results/data/figures in this manuscript have not been published elsewhere, nor are they under consideration by another publisher.

\end{document}